\documentclass{amsart}
\usepackage{amsfonts}
\usepackage{enumerate}

\usepackage{amscd,amssymb,amsmath,graphicx,verbatim}
\usepackage[dvips]{hyperref}
\usepackage[TS1,OT1,T1]{fontenc}

\newtheorem{theorem}{Theorem}[section]
\newtheorem{lemma}[theorem]{Lemma}
\newtheorem{corollary}[theorem]{Corollary}
\newtheorem{proposition}[theorem]{Proposition}

\theoremstyle{definition}
\newtheorem{definition}[theorem]{Definition}
\newtheorem{example}[theorem]{Example}

\newtheorem{question}[theorem]{Question}
\newtheorem{claim}[theorem]{Claim}
\theoremstyle{remark}
\newtheorem{remark}[theorem]{Remark}

\numberwithin{equation}{section}

\begin{document}

\title{Cowen-Douglas Operator and Shift on Basis}

\author{Juexian Li}
\address{}
\curraddr{School of Mathematics, Liaoning University, Shenyang 110036, People¡¯s Republic
of China}
\email{juexianli@sina.com}
\thanks{This project is supported by the National Natural Science Foundation of China (Grant No.
11371182, 11401283 and 11271150).}

\author{Geng Tian}
\address{School of Mathematics, Liaoning University, Shenyang 110036, People¡¯s Republic
of China and Department of Mathematics, Texas A M University, College Station,
TX 77843}
\email{xiluomath@sina.com}

\author{Yang Cao}
\address{School of Mathematics, Jilin University, Changchun 130012, People¡¯s Republic of
China}
\email{caoyang@jlu.edu.cn}

\subjclass[2000]{Primary 54C40, 14E20; Secondary 46E25, 20C20}



\keywords{Cowen-Douglas operator, Basis theory, Complex bundle}

\begin{abstract}
In this paper we show a Cowen-Douglas
operator $T \in \mathcal{B}_{n}(\Omega)$ is the adjoint operator of some backward shift on a general
basis by choosing nice cross-sections of its complex bundle $E_{T}$. Using the
basis theory model, we show that a Cowen-Douglas operator never be a shift on some
Markushevicz basis for $n \ge 2$.
\end{abstract}

\maketitle

\section{Introduction}

In this paper we try to make a basis theory understanding of Cowen-Douglas operators.
Let $\mathcal H$ be a separable, infinite dimensional, complex
Hilbert space, and  let ${\mathcal L}({\mathcal H})$ denote the
algebra of all bounded linear operators on $\mathcal H$. For $T\in
{\mathcal L}({\mathcal H})$, $T^*$, $\sigma (T)$ and $r(T)$ denote
the adjoint of $T$, the spectrum of $T$ and the spectral radius of
$T$, respectively.

For a connected open subset $\Omega $ of the complex plane
$\mathbb{C}$ and a positive integer $n$, ${\mathcal B}_n(\Omega )$
denotes the set of operators $T$ in ${\mathcal L}({\mathcal H})$
which satisfy

(1) $\Omega \subset \sigma (T)$;

(2) $\mathrm{ran}(T-z)={\mathcal H}$, $\forall z \in \Omega $;

(3) $\bigvee_{z\in\Omega}{\rm ker}(T-z)=\mathcal H$, and

(4) $\mathrm{dim\,ker}(T-z )=n$, $\forall z \in \Omega $.

\noindent Call an operator in ${\mathcal B}_n(\Omega)$ a
Cowen-Douglas operator \cite[Definition 1.2]{Douglas1}.

Clearly, if $T\in {\mathcal B}_n(\Omega )$ then $\Omega \subset \rho
_F(T)$ which denotes Fredholm domain and
$$
\mathrm{ind}(T-z)=\mathrm{dim\,ker}(T-z)-\mathrm{dim\,ker}(T-z)^*=n
\hbox{ for } z \in \Omega.
$$

For $T\in{\mathcal B}_n(\Omega)$, the mapping $z\mapsto\ker(T-z)$
defines a rank $n$ Hermitian holomorphic vector bundle, or briefly
complex bundle, $E_T$ over $\Omega$. A holomorphic cross-section of
the complex bundle $E_T$ is a holomorphic map
$\gamma:\Omega\rightarrow\mathcal H$ such that for every
$z\in\Omega$, the vector $\gamma(z)$ belongs to the fibre
$\ker(T-z)$ of $E_T$. Moreover, for a complex number $z_0$ in
$\Omega$, we will also consider a local cross-section of $E_T$
defined on a a neighborhood $\Delta$ of $z_0$.

Originally, Cowen-Douglas operators were introduced as using the
method of complex geometry to developing operator theory (see
\cite{Douglas1} and \cite{Douglas2}). However, it has been presented
recently that they are closely related to the structure of bounded
linear operators (see \cite{Jiang1}, \cite{Jiang2}, \cite{Jiang3}
and \cite{Jiang4}).

A typical example for a $n$-multiplicity Cowen-Douglas operator is
$n$-multiplicity backward shift on an orthonormal basis. A
characterization  $n$-multiplicity backward operator weighted shifts
being Cowen-Douglas operators has ever given in terminology of their
weight sequences \cite{Li}. In this paper, we shall show the
following.

\begin{theorem}\label{Theorem: CD Operators are shift on C-M sequence}
Let $T \in \mathcal{B}_{n}(\Omega)$. Then for every $z_{0} \in
\Omega$, there exists a complete and minimal vector sequence
$\{f_{k}\}_{k=0}^{\infty}$ in $\mathcal{H}$ such that $T-z_{0}$ is
the backward shift on $\{f_{k}\}_{k=0}^{\infty}$.
\end{theorem}

Here the word ``complete'' means
$\bigvee_{k=0}^\infty\{f_{k}\}=\mathcal{H}$, that is the linear
compositions of $\{f_{k}\}_{k=0}^{\infty}$  are dense in
$\mathcal{H}$. And, The word ``minimal'' refers to $f_{n} \notin
\bigvee_{k \not= n}\{f_{k}\}$.

By Hahn-Banach Theorem, a sequence $\{f_{k}\}_{k=0}^{\infty}$ in
$\mathcal{H}$ is minimal if and only if there is a sequence
$\{g_{k}\}_{k=0}^{\infty}$ in $\mathcal{H}$ such that
$$
(f_i,g_j)=\delta_{ij},~i,j=0,1,2,\cdots,
$$
i.e., the pair $(f_k,g_k)$ is a biorthogonal system. If
$\{g_{k}\}_{k=0}^{\infty}$ is also total, which means $\{x\in
\mathcal{H}:\ (x,g_k)=0,\ \forall k\geq 0\}=\{0\}$ (or equivalently,
$\{g_{k}\}_{k=0}^{\infty}$ is complete), then
$\{f_{k}\}_{k=0}^{\infty}$ is called  a \textit{generalized basis}
of $\mathcal{H}$ \cite[Def. 7.1]{Singer2}. For a generalized basis
$\{f_{k}\}_{k=0}^{\infty}$, if it is also complete then say it to be
a \textit{Markushevich basis} of $\mathcal{H}$ \cite[Def.
8.1]{Singer2}. Obviously, a Schauder basis is always a Markushevich
basis.

Use basis theory terminology,
we can get a stronger version of theorem \ref{Theorem: CD Operators are shift on C-M sequence}.

\begin{theorem}\label{Theorem: CD operator is an adjoint op of some shift on G basis}
For a Cowen-Douglas operator $T \in \mathcal{B}_{n}(\Omega), 0\in \Omega$, its adjoint operator
$T^{*}$ is a shift on some generalized basis.
\end{theorem}

Theorem \ref{Theorem: CD operator is an adjoint op of some shift on G basis} tell us that a Cowen-Douglas operator always is an adjoint of
some shift on a generalized basis when its spectrum contains $0$.

Weighted shifts on a a generalized basis (or a Markushevich basis)
of Banach space have ever investigated in \cite{Grabiner}.
The following result shows what happen for the operators in $\mathcal{B}_{n}(\Omega)$ in the case $n\ge 2$.

\begin{theorem}\label{Theorem: A Cowen-Douglas oper never be a shift on basis for n ge 2}
Let $T \in{\mathcal L}({\mathcal H})$. If there is a complex number
$z_0$ such that $\mathrm{dim\,ker}(T-z_0)\geq 2$ (in particular, if
$T \in \mathcal{B}_{n}(\Omega)$ and $n\ge 2$) then $T$ never is a
backward shift on any  Markushevich basis of $\mathcal{H}$.
\end{theorem}

Therefore a Cowen-Douglas operator $T \in \mathcal{B}_{n}(\Omega), n\ge 2$ never be a shift on some Schauder basis.
On the other hand, example \ref{Example: Shift on conditional basis} shows that
if we consider the n-multiple shift case(or more general, operator-weighted shift situation) then a Cowen-Douglas operator
(In this example, we can choose the canonical shift $(S^{*})^{2}$)can be
seen as a 2-multiple backward shift on a conditional basis.

Although the aim of this paper is to show the relations between basis theory and the class of
Cowen-Douglas operator, the main tools to get the proper sequences is choosing the good cross-sections
of the complex bundle $E_{T}$. In the next section we shall recall some basic results about
the special cross-section of the complex bundle of
Cowen-Douglas operators. In the third section we prove our main theorem \ref{Theorem: CD Operators are shift on C-M sequence}.
After this
we shall propose a general shift on biorthonal system model for Cowen-Douglas operators
in the lemma \ref{Lemma: Adjoint of a Shift on a spanning minimal seq must be a shift on G basis}
and prove theorem \ref{Theorem: CD operator is an adjoint op of some shift on G basis}
and theorem \ref{Theorem: A Cowen-Douglas oper never be a shift on basis for n ge 2}.
In the last section, we focus on the case $\mathcal{B}_{1}(\Omega)$. Theorem \ref{Theorem: when Cowen-Douglas operators are shifts}
give some equivalent conditions to decide whether a Cowen-Douglas operator $T \in \mathcal{B}_{1}(T)$
is a shift on some Markushevicz basis or not. And then theorem \ref{Theorem: C-D operators are shifts (1)}
give an operator theory description of the condition that  a Cowen-Douglas operator $T \in \mathcal{B}_{1}(T)$
can be a backward
weighted shift on an ONB.

\section{Canonical Cross-sections of The Complex Bundles of Cowen-Douglas Opertaotrs}

Firstly, we figure out a special operator related to every
Cowen-Douglas operator $T$ in $\mathcal{B}_{n}(\Omega)$. It can be
used to build special cross-sections of the complex bundle $E_{T}$.

\begin{lemma}\label{Lemma: Existence of Canonical Right Inverse}
Let $T\in{\mathcal L}({\mathcal H})$ and $ran T={\mathcal H}$. Then
there is an operator $B$ in ${\mathcal L}({\mathcal H})$ such that
$TB=I$ and $ran B=(ker T)^\bot$. Moreover, if an operator $B_1$ in
${\mathcal L}({\mathcal H})$ satisfies $TB_1=I$ and $ran B_1 \subset
(ker T)^\bot$ then $B_1=B$.
\end{lemma}

\begin{proof}
Let $\mathcal M=({\rm ker}T)^\bot$ and let $T_{\mathcal M}:\mathcal
M\rightarrow\mathcal H$ be the restriction of $T$ on $\mathcal M$,
i.e. $T_{\mathcal M}x=Tx$ for all $x\in\mathcal M$. Since
$T_{\mathcal M}$ is bijective, there is exactly a bounded linear
operator $B:\mathcal H\rightarrow\mathcal M$ such that $T_{\mathcal
M}B=I:=I_\mathcal H$ and $BT_\mathcal M=I_\mathcal M$. Thus, it is
easy to verify that $TB=I$ and ${\rm ran} B=\mathcal M=({\rm ker}
T)^\bot$.

Assume that $B_1$ is another operator satisfying $TB_1=I$ and ${\rm
ran} B_1 \subset({\rm ker} T)^\bot$. Then we have that $T(B-B_1)=0$.
Thus $(B-B_1)x$ is in ${\rm ker} T\cap({\rm ker} T)^\bot$, so
$Bx=B_1x$ for all $x\in\mathcal H$, that is $B=B_1$.
\end{proof}

In the sequel, we shall call the operator $B$ in Lemma \ref{Lemma:
Existence of Canonical Right Inverse} the {\it canonical right
inverse} of the operator $T$.

\begin{theorem}\label{Theorem: s_u is a cross-section of complex bundle of_T}
Let $T \in {\mathcal B}_n(\Omega)$, $z_0\in\Omega$ and let $B$ be
the canonical right inverse of $T-z_{0}$. Then, for every unit
vector $u$ in $\ker (T-z_{0})$, the  $\mathcal{H}$-valued
holomorphic function $s_u$ defined by
$$
s_{u}(z)=u+\sum_{k=1}^{\infty} B^{k}u\cdot (z-z_{0})^{k}
$$
is a local cross-section of the complex bundle $E_{T}$ on a
neighborhood $\Delta$ of $z_0$.
\end{theorem}
\begin{proof}
Clearly the power series converges on an open disc $\Delta$ with the
center at $z_{0}$. Moreover we have

\begin{eqnarray*}
(T-z_{0})s_{u}(z)&=&(T-z_{0})\Big(u+\sum_{k=1}^{\infty} B^{k}u\cdot (z-z_{0})^{k}\Big) \nonumber\\
         &=&(T-z_{0})u+(T-z_{0})\Big(\sum_{k=1}^{\infty} B^{k}u\cdot (z-z_{0})^{k}\Big)\nonumber\\
         &=&(z-z_{0})\sum_{k=0}^{\infty} B^{k}u\cdot (z-z_{0})^{k}\nonumber \\
         &=&(z-z_{0})s_{u}(z).
\end{eqnarray*}

Hence, $Ts_{u}(z)=zs_{u}(z)$. That is, the vector $s_{u}(z)$ belongs
to $\ker(T-z)$ for all $z\in\Delta$.
\end{proof}

With above notations, we have

\begin{definition}\label{Definition: Canonical cross-section}
The cross-section $s_{u}$ is called the \it{canonical cross-section}
with the initial unit vector $u$ at the point $z_{0}$. Moreover, the
n-tuple $\{s_{e_{1}}, s_{e_{2}}, \cdots, s_{e_{n}}\}$ will be called
the \it{canonical n-tuple} related to an orthonormal basis
$\{e_{1},\cdots ,e_{n}\}$ of $\ker (T-z_{0})$ at the point $z_{0}$.
\end{definition}

\begin{remark}
Let $B$ be the canonical right inverse of $T-z_{0}$. Since ${\rm
ran} B=(\ker (T-z_{0}))^\bot$, it is easy to show that the family
$\{s_{e_{1}}(z), s_{e_{2}}(z), \cdots, s_{e_{n}}(z)\}$ is linear
independent for every $z$ near $z_{0}$.  Thus, it forms a basis of
$\ker(T-z)$.
\end{remark}

As an application of Lemma \ref{Lemma: Existence of Canonical Right
Inverse} and Theorem \ref{Theorem: s_u is a cross-section of complex
bundle of_T}, we will show the following corollary, which was raised
by M. J. Cowen and R. G. Douglas in \cite{Douglas1}. However, they
did not give a proof.

\begin{corollary}
In the difinition of Cowen-Douglas operator, the condition (3) can
be equivalently replaced by the condition
$\bigvee_{k=1}^{\infty}ker(T-z_0)^k=\mathcal H$ for a fixed $z_0$ in
$\Omega$.
\end{corollary}
\begin{proof}
Assume $\bigvee_{z\in\Omega}ker(T-z)=\mathcal H$. Take an
orthonormal basis $\{e_1,\ldots,e_n\}$ of ${\rm ker}(T-z_0)$. We
know that the vector family $\{s_{e_1}(z),\ldots,s_{e_n}(z)\}$ is a
basis of ${\rm ker}(T-z)$ for every $z\in \Delta$. Note that
$(T-z_0)^{k+1}B^ke_i=(T-z_0)e_i=0$, we have that $B^ke_i\in{\rm
ker}(T-z_0)^{k+1}$ for all $k\geq0$ and $i=1,2,\ldots,n$.

Now, let $\mathcal M=\bigvee_{k=1}^{\infty}{\rm ker}(T-z_0)^k$. If
$x\in\mathcal H$ and $x\bot M$, then $x\bot B^ke_i$. It follows that
$x\bot s_{e_i}(z)$ for $1 \leq i \leq n$ and $z\in\Delta$. Thus, we
have that $x\bot {\rm ker}(T-z)$ for $z\in\Delta$. By
\cite[Corollary 1.13]{Douglas3}, we know that
$$
\bigvee_{z\in\Delta}{\rm ker}(T-z)=\bigvee_{z\in\Omega}{\rm
ker}(T-z)=\mathcal H.
$$
So, $x=0$. This shows that $\mathcal M=\mathcal H$.

Conversely, assume that $\bigvee_{k=1}^{\infty}{\rm
ker}(T-z_0)^k=\mathcal H$. If $x\in\mathcal H $ and
$x\bot\bigvee_{z\in\Delta}{\rm ker}(T-z)$, then $x\bot e_i$ and
$x\bot s_{e_i}(z)$ for all $z\in\Delta$ and $1 \leq i \leq n$. Thus,
we obtain that
$$
0=(x,s_{e_i}(z))=(x,e_i)+\sum_{k=1}^{\infty}(x,B^ke_i)(z-z_0)^k,\,\forall
z\in\Delta.
$$
Hence, we have that $(x,B^ke_i)=0$ for all $k\geq0$ and
$i=1,2,\ldots,n$. Since $e_i\in{\rm ker}(T-z_0)$ and
$(T-z_0)B^ke_i=B^{k-1}e_i$, we can show by induction that, for each
$k\geq0$, the vector family
$$\{e_1,\ldots,e_n,\ldots,B^ke_1,\ldots,B^ke_n\}$$
in ${\rm ker}(T-z_0)^{k+1}$ is linearly independent. Note that
$\mathrm{dim\,ker}(T-z_0)^{k+1}=(k+1)n$, it follows that this family
is exactly a basis  of ${\rm ker}(T-z_0)^{k+1}$, which proves that
$x\bot{\rm ker}(T-z_0)^{k+1}$ for all $k\geq0$. So, $x=0$.
\end{proof}

From Definition \ref{Definition: Canonical cross-section}, we know
that there are many canonical n-tuples dependent on the choice of an
orthonormal basis $\{e_{1},e_{2},\cdots ,e_{n}\}$ of $\ker (T-z_0)$.
However, following lemma tells us that the class of canonical
n-tuples is small.

\begin{lemma}\label{Lemma: Canonical Cross-sections are not So Many}
Assume that $(s_{1},\cdots ,s_{n})$ and $(\tilde{s}_{1},\cdots
,\tilde{s}_{n})$ are canonical section tuples of $T \in {\mathcal
B}_n(\Omega)$ related to ONB $\{e_{1},e_{2},\cdots ,e_{n}\} $ and
$\{\tilde{e}_{1},\tilde{e}_{2},\cdots ,\tilde{e}_{n}\} $
respectively, then there is an unique unitary matrix $U \in U(n)$
such that
$$
(\tilde{s}_{1},\cdots ,\tilde{s}_{n})=(s_{1},\cdots ,s_{n})U.
$$
\end{lemma}

\begin{proof}It is clear that there is an unitary
matrix $U=(u_{ij})_{n \times n} \in U(n)$ such that
$$
 (\tilde{e}_{1},\tilde{e}_{2},\cdots ,\tilde{e}_{n})
=(e_{1},e_{2},\cdots,e_{n})U.
$$
Then we have
$$
\begin{array}{ll}
\tilde{s}_{i} & =\sum_{k=1}^{\infty} \tilde{e}_{i}+ B^{k}\tilde{e}_{i}                  \\
          & =\sum_{k=1}^{\infty} (\sum_{j=1}^{n} u_{ji}e_{j})
            +B^{k}(\sum_{j=1}^{n} u_{ji}e_{j})                               \\
          & =\sum_{j=1}^{n} u_{ji}(\sum_{k=1}^{\infty}e_{j}+B^{k}e_{j})     \\
          & =\sum_{j=1}^{n} u_{ji}s_{j} ,
\end{array}
$$
Or equivalently, we have
$(\tilde{s}_{1},\cdots,\tilde{s}_{n})=(s_{1},\cdots ,s_{n})U$.
\end{proof}

\begin{lemma}\label{Lemma: Uniqueness of the Canonical Cross-section}
Let $T \in {\mathcal B}_n(\Omega)$, $z_0\in\Omega$ and let $u$ be a
unit vector in $\ker (T-z_{0})$. Then there is exactly one
holomorphic cross-section $\gamma$ defined on a neighborhood
$\Delta$ of $z_0$ such that $\gamma(z_0)=u$ and
$\gamma(z)-u\in(\ker(T-z_0))^\bot$ for all $z\in\Delta$.
\end{lemma}
\begin{proof}
By Theorem \ref{Theorem: s_u is a cross-section of complex bundle
of_T}, the canonical cross-section $s_u$ satisfies required
properties. For uniqueness, suppose that there is another
holomorphic cross-section $\gamma$ which has required properties.
Then let $\gamma$ have the following power series expansion
$$
\gamma(z)=u+\sum_{k=1}^{\infty} u_{k}(z-z_{0})^{k}.
$$
Since $(T-z_0)\gamma(z)=(z-z_0)\gamma(z)$, it can be checked that $
(T-z_{0})u=0$ and $(T-z_{0})u_{k}=u_{k-1}$ for $k\ge 1,$ where
$u_0=u$. Also, let $B\in{\mathcal L}({\mathcal H})$ be the canonical
right inverse of $T-z_0$. Then we have $Bu_{k-1}-u_{k} \in \ker
(T-z_{0})$. Note that
$$
\dfrac{\gamma(z)-u}{z-z_0}=\sum_{k=1}^{\infty}
u_{k}(z-z_{0})^{k-1}\in(\ker(T-z_0))^\bot,
$$
we obtain that
$$
u_1=\lim_{z\rightarrow
z_0}\dfrac{\gamma(z)-u}{z-z_0}\in(\ker(T-z_0))^\bot.
$$
Thus, $u_k\in(\ker(T-z_0))^\bot$ for all $k\geq1$ by induction. Now,
$$
Bu_{k-1}-u_{k} \in (\ker(T-z_{0}))^{\bot}\cap \ker(T-z_{0}),
$$
that is $Bu_{k-1}=u_{k}$. This implies $s_u(z)=\gamma(z)$ on some
neighborhood of $z_0$.
\end{proof}

In light of Lemma \ref{Lemma: Uniqueness of the Canonical
Cross-section}, we can give an equivalent description of the
canonical cross-section as follows.

\begin{proposition}\label{Definition: Eq Def of CS}
Let $T \in {\mathcal B}_n(\Omega)$, $z_0\in\Omega$ and let $u$ be a
unit vector in $\ker (T-z_{0})$. Then a holomorphic cross-section
$s_u$ of $E_{T}$ defined on a neighborhood $\Delta$ of $z_0$ is the
canonical cross-section with the initial vector $u$ if and only if
$s_u(z_0)=u$ and $s_u(z)-u\in(\ker(T-z_0))^\bot$ for all
$z\in\Delta$.
\end{proposition}

In what follows let $\mathcal{G}r(n, \mathcal{H})$ denote the
Grassmann manifold consisting of all $n$-dimensional subspaces of
$\mathcal{H}$. A map $f: \Omega \rightarrow \mathcal{G}r(n,
\mathcal{H})$ is called holomorphic at $z_0\in\Omega$ if there
exists a neighborhood $\Delta$ of $z_0$ and $\mathcal{H}$-value
functions $\gamma_1,\cdots,\gamma_n$ such that
$f(z)=\bigvee\{\gamma_1(z),\cdots,\gamma_n(z)\}$ for $z\in\Delta$.
If $f: \Omega \rightarrow \mathcal{G}r(n, \mathcal{H})$ is
holomorphic map then $f$ induces a natural complex bundle $E_f$ as
follows:
$$
E_f=\{(x,z)\in\mathcal{H}\times\Omega:\,x\in f(z)\}.
$$
And, the projection $\pi:E_f\rightarrow\Omega$ is given by
$\pi(x,z)=z$. The bundle $E_f$ will be called the pull-back bundle
of the Grassmann manifold induced by $f$. In particular, for an
operator $T \in {\mathcal B}_n(\Omega)$, we can define $f: \Omega
\rightarrow \mathcal{G}r(n, \mathcal{H})$ by $f(z)=\ker(T-z)$ for
$z\in\Omega$, then $f$ is holomorphic and $E_f=E_T$.

The existence of the canonical cross-section is just a special case
of \cite[Lemma 2.4]{Douglas3} when $E_f=E_T$. We can prove that
Lemma \ref{Lemma: Uniqueness of the Canonical Cross-section} also
holds for the general holomorphic map.

\begin{theorem}\label{Theorem: Genaral Situation of Canonical Cross-Section}
Let $E_{f}$ is the pull-back bundle of $\mathcal{G}r(n,
\mathcal{H})$ induced by a holomorphic map $f: \Omega \rightarrow
Gr(n, \mathcal{H})$, and let $\{u_{1}, u_{2}, \cdots, u_{n}\}$ is an
orthonormal basis of the fibre $f(z_{0})$. Then there exist exactly
$n$ holomorphic cross-sections $\gamma_{1}, \gamma_{2}, \cdots,
\gamma_{n}$ of $E_{f}$ defined on some open disc $\Delta$ with the
center at $z_{0}$ such that

(1) $\gamma_{1}(z), \cdots, \gamma_{n}(z)$ form a basis of the fibre
$f(z)$ for $z\in \Delta$;

(2) $\gamma_{k}(z_{0})=u_{k}$ for $1\le k\le n$; and

(3) $(\gamma_{k}(z)-\gamma_{k}(z_{0}), \gamma_{j}(z_{0}))=0$  for
$1\le k, j\le n$ and  $z\in \Delta$.
\end{theorem}
\begin{proof}
It is easy to show that there exist $n$ holomorphic cross-sections
on a neighborhood of $z_{0}$ with the initial vector $u_{k}$  for
all $1\le k\le n$ respectively. Thus, the existence of holomorphic
cross-sections satisfying the properties (1), (2) and (3) is
directly from \cite[Lemma 2.4]{Douglas3}. For the uniqueness, it is
enough to prove

\begin{claim} For the unit vector $u_{1}$, the
cross-section $\gamma_{1}$ is the unique holomorphic cross-section
satisfying properties (1), (2) and (3).
\end{claim}

Assume that $\gamma(z)=u_{1} + \sum_{i=1}^{\infty}
a_{i}(z-z_{0})^{i}$ is another holomorphic cross-section defined on
$\Delta$ which satisfies properties (1), (2) and (3). And, let
$$
\gamma_{k}(z)=u_{k} +\sum_{i=1}^{\infty} a^{(k)}_{i}(z-z_{0})^{i}
$$
be the power series expansions of $\gamma_{k}$ on $\Delta$ for $1\le
k\le n$. By the property (2) we have
\begin{equation}\label{Equation: Orthogonal of Vectors Appearing in Canonical Secs}
(a_{i}, u_{j})=(a^{(k)}_{i}, u_{j})=0 \hbox{ for } i \ge 1 \hbox{
and } 1 \le j \le n.
\end{equation}
Now there are holomorphic functions $h_{1}(z), \cdots, h_{n}(z)$
defined $\Delta$ such that
$$
\gamma(z)=h_{1}(z)\gamma_{1}(z)+h_{2}(z)\gamma_{2}(z)+\cdots +h_{n}(z)\gamma_{n}(z)
$$
for all $z\in \Delta$. Clearly we have $h_{k}(z_{0})=0$ for $2 \leq
k \le n$ and $h_{1}(z_{0})=1$. Suppose that $h_{k}(z)\not\equiv 0$
for $2\le k\le n$. Then we can write
$$
h_{k}(z)=(z-z_{0})^{m}\tilde{h}_{k}(z), \hbox{ where }
\tilde{h}_{k}(z)=c_{0}+\sum_{i=1}^{\infty} c_{i}(z-z_{0})^{i} \hbox{
and } c_{0}\ne 0.
$$
However, by the product of power series, we can obtain that
$$a_m=c_0u_k+v.$$ By the formula $(2.1)$ and noting that $u_k\bot
u_l(k\neq l)$, we have that $u_k\bot v$. So, still from the formula
$(2.1)$, it follows that $$c_0=(c_0u_k,u_k)=(a_m,u_k)=0.$$ This
contradicts $c_0\neq0$. Now, we have that $\gamma=h_1\gamma_{1}$.
Let$$h_1(z)=1+\sum_{i=1}^\infty d_i(z-z_0)^i.$$ Then it follows that
$$a_1=a_1^{(1)}+d_1u_1.$$ Again, by the formula $(2.1)$, we get that
$$0=(a_1,u_1)=(a_1^{(1)},u_1)+(d_1u_1,u_1)=d_1.$$ So, $a_1=a_1^{(1)}$.
By induction, we have that $d_i=0$ for $i\geq 1$, that is
$\gamma=\gamma_1$.
\end{proof}

Above theorem say that together with lemma 2.4 of paper \cite{Douglas3} we can apply our lemma \ref{Lemma: Uniqueness of the Canonical Cross-section}
to get a more elegant way to obtain the canonical cross-section from a initial vector. We place that method in the next subsection for
readability.


\section{Cowen-Douglas operators are shifts on complete minimal sequences}

\subsection{Psedocanonical cross-sections}

$\quad$

Suppose $T \in \mathcal{B}_{n}(\Omega)$ and $\gamma$ be a
holomorphic cross-section of $E_{T}$ defined on a connected open
subset  $\Delta$ of $\Omega$. Call $\gamma$ to be {\it spanning }if
$\bigvee\{\gamma(z):\,z\in \Delta \}=\mathcal{H}$

Let $\gamma(z)=f_{0}+\sum_{k=1}^{\infty} f_{k}(z-z_{0})^{k}$ be the
power series expansion of $\gamma$ on an open disc $\Delta$ with the
center at $z_{0}$. Clearly, $\gamma$ is spanning if and only if
$\bigvee_{k=0}^\infty\{f_{k}\}=\mathcal{H}$, i.e., the vector
sequence $\{f_{k}\}_{k=0}^{\infty}$ is complete.

\begin{theorem}\label{Theorem: Holomorphic cross-section is spanning}
If $T \in \mathcal{B}_{1}(\Omega)$ then every holomorphic
cross-section $\gamma$ is spanning.
\end{theorem}
\begin{proof}
Suppose that  $x \perp \gamma(z)$ for all $z \in\Delta$. It follows
that $x \perp\bigvee_{z\in\Delta}{\rm ker}(T-z)$ since
$\mathrm{dim\,ker}(T-z )=1$. By \cite[Corollary 1.13]{Douglas1}, we
know that $\bigvee_{z\in\Delta}{\rm ker}(T-z)=\mathcal H$. So,
$x=0$. This shows that $\bigvee\{\gamma(z):\,z\in \Delta
\}=\mathcal{H}$
\end{proof}

However, when $n>1$ any canonical cross-section
$\gamma(z)=f_{0}+\sum_{k=1}^{\infty} f_{k}(z-z_{0})^{k}$ never be
spanning since $f_{0}\in\ker(T-z_0)$ and
$f_{k}\in(\ker(T-z_0))^\bot$ for $k \geq 1$. To generalize above
theorem to the case $\mathcal{B}_{n}(\Omega)$, we need following
notion.

\begin{definition}\label{Definition: Psedocanonical cross section}
Let $T \in {\mathcal B}_n(\Omega)$ and $z_0\in\Omega$. A holomorphic
cross-section $\mu$ of the complex bundle $E_{T}$ defined on a
neighborhood $\Delta$ of $z_{0}$ is said to be psedocanonical if
$\mu(z_0)\not=0$ and $(\mu(z_{0}), \mu'(z))=0$ for all $z\in\Delta$.
\end{definition}

Suppose that $\mu(z)=f_{0}+\sum_{n=1}^{\infty} f_{n}(z-z_{0})^{n}$
is the power series expansion of $\mu$ at $z_0$. Then it is obvious
that $\mu$ is psedocanonical if and only if $f_{0}\bot f_k$ for all
$k\geq1$.

Clearly, a canonical cross-section  must be psedocanonical. If $T
\in \mathcal{B}_{1}(\Omega)$, by Proposition \ref{Definition: Eq Def
of CS}, they are equivalent.

\begin{proposition}\label{Proposition: Construction of psedocan cross section}
Let $T \in \mathcal{B}_{n}(\Omega)$ and  $\mu$ be a psedocanonical
cross-section, and let $\gamma$ be a canonical cross-section which
satisfy $(\mu(z_{0}), \gamma(z_{0}))=0$. Then for any holomorphic
function $g(z)$ defined near $z_{0}$ with $g(z_{0})=0$, the
cross-section $\lambda(z)=\mu(z)+g(z)\gamma(z)$ is also
psedocanonical.
\end{proposition}
\begin{proof}
By the definition of canonical cross-section, we have
$$
(\mu(z_{0}), \gamma(z)-\gamma(z_{0}))=0 \hbox{ and } (\mu(z_{0}),
\gamma'(z))=0
$$
since $\mu(z_{0}) \in \ker(T-z_{0}I)$. Hence, we have $(\mu(z_{0}),
\gamma(z))=(\mu(z_{0}), \gamma(z_{0}))=0$. Note that
$\lambda(z_{0})=\mu(z_{0})$, it follows that
\begin{eqnarray*}
(\lambda(z_{0}), \lambda'(z))&=&(\mu(z_{0}), \mu'(z)+g'(z)\gamma(z)+g(z)\gamma'(z)) \\
                             &=&(\mu(z_{0}), \mu'(z))+(\mu(z_{0}), g'(z)\gamma(z))+(\mu(z_{0}), g(z)\gamma'(z)) \\
                             &=&0.
\end{eqnarray*}
\end{proof}

\begin{theorem}\label{Theorem: Structure of Psedocanonical cross-section}
Let $T \in \mathcal{B}_{n}(\Omega)$, $z_{0} \in \Omega$ and let
$\lambda$ be a psedocanonical cross-section of $E_{T}$. Then there
are canonical cross-sections $\gamma_{1},\cdots, \gamma_{n}$ and
holomorphic functions $g_{2}, \cdots, g_{n}$ defined near $z_0$ such
that

(1) $(\gamma_{i}(z), \gamma_{j}(z_{0}))=\delta_{ij}$ for $1 \leq i,j
\leq n$ and  $\gamma_{1}(z_{0})=\lambda(z_{0})$;

(2) $g_{i}(z_{0})=0$ for $i=2, 3, \cdots, n$; and

(3) $\lambda(z)=\gamma_{1}(z)+\sum_{i=2}^{n} g_{i}(z)\gamma_{i}(z)$.

\end{theorem}
\begin{proof}
Clearly we can assume $||\lambda(z_{0})||=1$. Take an orthonormal
basis $\{e_{1}, e_{2}, \cdots, e_{n}\}$ of $\ker(T-z_{0})$ such that
$e_{1}=\lambda(z_{0})$. From Definition \ref{Definition: Canonical
cross-section}, we set $\gamma_i=s_{e_i}$ which is the canonical
cross-section with the initial vector $\gamma_i(z_0)=e_i$ for
$i=1,2, \cdots, n$ respectively. Since
$\gamma_i(z)-\gamma_i(z_0)\bot\ker(T-z_0)$, we can obtain that
$$
(\gamma_i(z),\gamma_j(z_0))=(\gamma_i(z)-\gamma_i(z_0),\gamma_j(z_0))+(\gamma_i(z_0),\gamma_j(z_0))=(e_i,e_j)=\delta_{ij},
$$
that is the property $(1)$ holds.

Moreover, let
\begin{eqnarray}
g_{i}(z)=(\lambda(z), \gamma_{i}(z_{0})), ~ i=2,\cdots,n.
\end{eqnarray}
Then $g_i(z_0)=(e_1,e_i)=0$, that is the property $(2)$ is also
true.

And, we define $\mu(z)$ by
\begin{eqnarray}
\mu(z)=\lambda(z)-\sum_{i=2}^{n} g_{i}(z)\gamma_i(z).
\end{eqnarray}
Then, by the formula $(3.1)$ and the property $(1)$, we have that
\begin{eqnarray*}
(\mu(z), \gamma_{j}(z_{0})) & =&(\lambda(z),\gamma_j(z_0))-\sum_{i=2}^{n} g_{i}(z)(\gamma_{i}(z), \gamma_{j}(z_{0})) \\
                            & =&(\lambda(z),
                            \gamma_{j}(z_{0}))-g_j(z)=0.
\end{eqnarray*}
Hence, for $i=2, \cdots, n$, it follows that
$$
(\mu(z)-\lambda(z_{0}),\gamma_i(z_0))=(\mu(z),\gamma_i(z_0))-(\lambda(z_{0}),\gamma_i(z_0))=(e_1,e_i)=0.
$$
Also, by the formula $(3.2)$ and Definition \ref{Definition:
Psedocanonical cross section}, we can show that
\begin{eqnarray*}
(\mu(z)-\lambda(z_{0}), \gamma_{1}(z_{0})) & =&(\lambda(z)-\lambda(z_0),\gamma_1(z_0))-\sum_{i=2}^{n} g_{i}(z)(\gamma_{i}(z), \gamma_{1}(z_{0})) \\
                            & =&(\lambda(z)-\lambda(z_0),
                            \lambda(z_0))=0\\
\end{eqnarray*}
since $\lambda$ is psedocanonical. This implies that
$$
(\mu(z)-\lambda(z_{0})) \perp \ker(T-z_{0}).
$$
Note that $\lambda(z_0)=\mu(z_0)=\gamma_1(z_0)=e_1$, by Lemma
\ref{Lemma: Uniqueness of the Canonical Cross-section}, we have
$\mu=\gamma_1$. Thus, the property $(3)$ holds.
\end{proof}

\begin{lemma}\label{Lemma: Obtain CS in CD Method}
Let
$$
\lambda(z)=f_{0}+\sum_{k=1}^{\infty} f_{k}(z-z_{0})^{k}
$$
be a holomorphic cross-section of $E_{T}$ defined near $z_{0}\in
\Omega$ with $f_{0}\not=0$. Then there is a unique holomorphic
function $h$ defined near $z_{0}$ such that $\mu(z)=h(z)\lambda(z)$
is a psedocanonical cross-section of $E_{T}$ at $z_{0}$ and
$\mu(z_{0})=f_{0}$.
\end{lemma}
\begin{proof}
We consider the holomorphic function $g(z)=(\lambda(z),f_0)$. Since
$g(z_0)=||f_0||^2\neq0$, we know $g(z)\neq0$ near $z_0$. Hence, the
function $h(z)=\frac{||f_0||^2}{g(z)}$ is holomorphic near $z_0$.

Now, let $\mu(z)=h(z)\lambda(z)$. Then we have
$(\mu(z),f_0)\equiv||f_0||^2$. It is easy to check that $\mu(z)$ has
the form of the power series expansion as
$$
\mu(z)=f_0+\sum_{k=1}^{\infty}c_k(z-z_{0})^{k}.
$$
Thus, we obtain that
$$\Big(\sum_{k=1}^{\infty}c_k(z-z_{0})^{k},f_0\Big)=(\mu(z)-f_0,f_0)=(\mu(z),f_0)-||f_0||^2\equiv0,$$
which implies that $f_0\bot c_k$ for $k\geq1$. Therefore, $\mu(z)$
is a psedocanonical cross-section of $E_{T}$ at $z_{0}$ and
$\mu(z_{0})=f_{0}$.

To show the uniqueness, suppose that
$$
\widetilde{h}(z)=f_0+\sum_{k=1}^{\infty}\widetilde{c_k}(z-z_{0})^{k}
$$
is another holomorphic function defined near $z_0$ such that
$\widetilde{\mu}(z)=\widetilde{h}(z)\lambda(z)$ is also a
psedocanonical cross-section and
$\widetilde{\mu}(z_0)=\widetilde{h}(z_0)\lambda(z_0)=f_0$. Since
$f_0\bot c_k$ and $f_0\bot \widetilde{c_k}$ for $k\geq1$, we have
that
$$((h(z)-\widetilde{h}(z))\lambda(z),f_0)=(\mu(z)-\widetilde{\mu}(z),f_0)\equiv0.$$
It implies that $h(z)\equiv\widetilde{h}(z)$.
\end{proof}

\begin{corollary}
Let $T \in \mathcal{B}_{n}(\Omega)$, $z_{0} \in \Omega$ and let
$\lambda$ be a holomorphic cross-section of $E_{T}$. Then there are
canonical cross-sections $\gamma_{1},\cdots, \gamma_{n}$ and
holomorphic functions $g_{1}, \cdots, g_{n}$ defined near $z_0$ such
that

(1) $(\gamma_{i}(z), \gamma_{j}(z_{0}))=\delta_{ij}$ for $1 \leq i,j
\leq n$ and  $\gamma_{1}(z_{0})=\lambda(z_{0})$;

(2) $g_{i}(z_{0})=0$ for $i=2, 3, \cdots, n$; and

(3) $\lambda(z)=\sum_{i=1}^{n} g_{i}(z)\gamma_{i}(z)$.
\end{corollary}
\begin{proof}
By Lemma \ref{Lemma: Obtain CS in CD Method}, we know that there is
a holomorphic function $h$ defined near $z_{0}$ such that
$\mu(z)=h(z)\lambda(z)$ is a psedocanonical cross-section of $E_{T}$
at $z_{0}$ and $\mu(z_{0})=\lambda(z_{0})$. Thus, it follows that
there are canonical cross-sections $\gamma_{1},\cdots, \gamma_{n}$
and holomorphic functions $\widetilde{g}_{2}, \cdots,
\widetilde{g}_{n}$ defined near $z_0$ such that the property $(1)$
holds, $\widetilde{g}_{i}(z_{0})=0$ for $i=2, 3, \cdots, n$, and
$$
 \mu(z)=h(z)\lambda(z)=\gamma_{1}(z)+\sum_{i=2}^{n} \widetilde{g}_{i}(z)\gamma_{i}(z).
$$
Since $h(z_0)=1$, the function $\frac{1}{h(z)}$ is holomorphic near
$z_0$. Now, set $g_{1}(z)=\frac{1}{h(z)}$ and
$g_{i}(z)=\frac{\widetilde{g}_{i}(z)}{h(z)}$. Then the properties
$(2)$ and $(3)$ hold.
\end{proof}

\subsection{Proof of main Theorem \ref{Theorem: CD Operators are shift on C-M sequence}}
\begin{theorem}\label{Theorem: Coefficent sequence of psedocanonical section is Minimal}
Let $T \in \mathcal{B}_{n}(\Omega)$, $z_{0} \in \Omega$ and let
$\mu(z)=f_{0}+\sum_{k=1}^{\infty} f_{k}(z-z_{0})^{k}$ be a
holomorphic cross-section defined near $z_{0}$ with $f_{0}\not= 0$.
Then the vector sequence $\{f_{k}\}_{k=0}^{\infty}$ is minimal if
and only if $f_{0} \not\in \bigvee_{k=1}^\infty\{f_{k}\}$. In
particular, when $\mu$ is a psedocanonical cross-section,
$\{f_{k}\}_{k=0}^{\infty}$ is always minimal.
\end{theorem}
\begin{proof}
We just need to show the part of ``if''. Let
$\mathcal{H}_{\mu}^{1}=\bigvee_{k=1}^\infty\{f_{k}\}$ and $f_{0}
\not\in \mathcal{H}_{\mu}^{1}$. Since $\mu(z)\in\ker(T-z)$ for all
$z$, it follows that
$$
\sum_{k=0}^\infty f_k(z-z_0)^{k+1}=(T-z_0)\mu(z)=\sum_{k=1}^\infty
(T-z_0)f_k(z-z_0)^{k}.
$$
This implies that $(T-z_0)f_k=f_{k-1}$ for $k\geq1$. Hence we have
$(T-z_0)^kf_k=f_0$.

If the statement is false, then there exist some $f_k$ and vector
sequence $\{v_n\}_{n=1}^\infty$ such that $v_n$ is a linear
combination of finite vectors in the set $\{f_n: n\neq k\}$ and
$||v_n-f_k||<\frac{1}{n}$. Moreover, we can write that
$$
v_{n}=v_{n}^{(1)}+v_{n}^{(2)}, v_{n}^{(1)}=\sum_{j=1}^{k-1}
\alpha_{j}f_{j}~ {\rm and} ~ v_{n}^{(2)}=\sum_{j=k+1}^{m_{n}}
\alpha_{j}f_{j}.
$$
Thus, we get that
$$
(T-z_0)^{k}v_{n}^{(1)}=0 ~{\rm and}~ (T-z_0)^{k}v_{n}^{(2)} \in
\mathcal{H}_{\mu}^{1}.
$$ It gives that $ (T-z_0)^{k}v_{n} \in
\mathcal{H}_{\mu}^{1}$ and
$$
||f_{0}-(T-z_0)^{k}v_{n}||=||(T-z_0)^{k}(f_k-v_{n})||\leq||T-z_0||^{k}\frac{1}{n}\rightarrow0(n\rightarrow\infty).
$$
Hence, it follows that $ f_{0} \in \mathcal{H}_{\mu}^{1}$, a
contradiction.
\end{proof}

Next, we can generalize Theorem \ref{Theorem: Holomorphic
cross-section is spanning} as follows.

\begin{theorem}\label{Theorem: Existence of complete minimal coes psedocanonical cross section}
Let $T \in \mathcal{B}_{n}(\Omega)$. Then for each $z_0\in\Omega$,
there is a spanning psedocanonical cross-section.
\begin{eqnarray}
\lambda(z)=f_{0}+\sum_{k=1}^{\infty} f_{k}(z-z_{0})^{k}
\end{eqnarray}
defined near $z_0$. Hence, the vector sequence
$\{f_{k}\}_{k=0}^{\infty}$ is complete.
\end{theorem}
\begin{proof}
By \cite[Theorem A]{Zhu}, there is a  holomorphic cross-section
$\widetilde{\lambda}$ on $\Omega$ such that
$\bigvee\{\widetilde{\lambda}(z):z\in\Omega\}=\mathcal{H}$ and
$\widetilde{\lambda}(z_0)\neq0$. Imitating the proof of
\cite[Corollary 1.13]{Douglas1}, for every open set $\Delta$ in
$\Omega$, we can show that
$\bigvee\{\widetilde{\lambda}(z):z\in\Delta\}=\mathcal{H}$.

Now, by Lemma \ref{Lemma: Obtain CS in CD Method}, there is a
holomorphic function $h(z)$ such that
$\lambda(z)=h(z)\widetilde{\lambda}(z)$ is a psedocanonical
cross-section with the initial vector
$f_0=\widetilde{\lambda}(z_0)$. Let the power series expansion of
$\lambda(z)$ on a disc $\Delta$ with the center at $z_0$ be given by
the formula $(3.3)$. Moreover, we can assume that $h(z)\neq0$ for
all $z\in\Delta$. Hence, we have that
$$
\bigvee\{\lambda(z):z\in\Delta\}=\bigvee\{\widetilde{\lambda}(z):z\in\Delta\}=\mathcal{H}.
$$
This implies that $\lambda$ is a spanning. Thus,
$\{f_{k}\}_{k=0}^{\infty}$ is complete.
\end{proof}

From Theorem \ref{Theorem: Coefficent sequence of psedocanonical
section is Minimal} and Theorem \ref{Theorem: Existence of complete
minimal coes psedocanonical cross section}, we obtain directly the
following.

\begin{theorem}\label{Theorem: General situation: Coefficent sequence of canonical section is Minimal}
Let $T \in \mathcal{B}_{n}(\Omega)$. Then for each $z_0\in\Omega$,
there is a psedocanonical cross-section
$$
\lambda(z)=f_{0}+\sum_{k=1}^{\infty} f_{k}(z-z_{0})^{k}
$$
such that the vector sequence $\{f_{k}\}_{k=0}^{\infty}$ is minimal
and complete.
\end{theorem}

Now we can apply above theorem to prove Theorem \ref{Theorem: CD
Operators are shift on C-M sequence}.

\textit{Proof of Theorem 1.1}. For every $z_{0} \in \Omega$. By
Theorem \ref{Theorem: General situation: Coefficent sequence of
canonical section is Minimal} there is a  psedocanonical
cross-section
$$
\lambda(z)=f_{0}+\sum_{k=1}^{\infty} f_{k}(z-z_{0})^{k}
$$
near $z_{0}$ such that the vector sequence
$\{f_{k}\}_{k=0}^{\infty}$ is both complete and minimal. Since
$\lambda(z) \in {\rm ker}(T-z)$, it follows that
$$
\sum_{k=0}^\infty
f_k(z-z_0)^{k+1}=(T-z_0)\lambda(z)=\sum_{k=1}^\infty
(T-z_0)f_k(z-z_0)^{k}.
$$
Hence, we have that $(T-z_0)f_k=f_{k-1}$ for $k\geq1$ and
$(T-z_{0})f_{0}=0$, that is $T-z_{0}$ is the backward shift on
$\{f_{k}\}_{k=0}^{\infty}$. $\hfill{}\Box$

\subsection{Proof of theorem \ref{Theorem: CD operator is an adjoint op of some shift on G basis}
 and theorem \ref{Theorem: A Cowen-Douglas oper never be a shift on basis for n ge 2}}


\begin{lemma}\label{Lemma: Adjoint of a Shift on a spanning minimal seq must be a shift on G basis}
Suppose $\{(f_{k}, g_{k})\}_{k=0}^{\infty}$ is a biorthogonal system and
$\{f_{k}\}_{k=0}^{\infty}$ is spanning, that is, $\vee_{k\ge 0}\{f_{k}\}=\mathcal{H}$.
Moreover, suppose $T$ is a backward shift on
the sequence $\{f_{k}\}_{k=0}^{\infty}$,
i.e., $Tf_{0}=0$ and $Tf_k=f_{k-1}$ for $k\geq1$. Then we have

\begin{enumerate}
\item
The sequence $\{g_{k}\}_{k=0}^{\infty}$ is a generalized basis of the Hilbert space $\mathcal{H}$;
\item
The adjoint operator $T^{*}$
must be a foreward shift on the generalized basis $\{g_{k}\}_{k=0}^{\infty}$, that
is, we have $T^{*}g_{k}=g_{k+1}$ for $k=0,1,\cdots.$
\end{enumerate}
\end{lemma}

\begin{proof}
\begin{enumerate}
\item
Firstly we show that $\{g_{k}\}_{k=0}^{\infty}$ is a generalized basis.
Since the sequence
$\{f_{k}\}_{k=0}^{\infty}$ is spanning, we have $(x, f_{k})=0$ for each $k\ge 1$ if and only if $x=0$.
Hence as a functional sequence, $\{f_{k}\}_{k=0}^{\infty}$ is total.
\item
Note that
$$
(f_{j+1},T^*g_i)=(Tf_{j+1},g_i)=(f_j,g_i)=\delta_{ij}.
$$
then
by $\{(f_{k}, g_{k})\}_{k=0}^{\infty}$ is a biorthogonal system, we know the sequence $\{f_{k}\}_{k=0}^{\infty}$
is minimal($f_{j+1} \notin \vee_{k \ne j+1, k\ge 0}\{f_{k}\}$). Moreover by
$\vee_{k\ge 0}\{f_{k}\}=\mathcal{H}$, we know that $\vee_{k \ne j+1, k\ge 0}\{f_{k}\}$ is a subspace
of the Hilbert space $\mathcal{H}$ with codimension $1$.
Therefore we must have that $ T^*g_k=g_{k+1}$ for all $k\geq 0$.
\end{enumerate}
\end{proof}

\textit{Proof of Theorem \ref{Theorem: CD operator is an adjoint op of some shift on G basis}}.
It is a directly result of above lemma \ref{Lemma: Adjoint of a Shift on a spanning minimal seq must be a shift on G basis}
and theorem \ref{Theorem: Existence of complete minimal coes psedocanonical cross section}.

\begin{lemma}\label{Lemma: Adjoint of a Shift on M basis also be a shift}
Suppose $\{(f_{k}, g_{k})\}_{k=0}^{\infty}$ is a biorthogonal system and
$\{f_{k}\}_{k=0}^{\infty}$ is a Markushevich basis.
Moreover, let $T$ be a backward shift on
the Markushevich basis $\{f_{k}\}_{k=0}^{\infty}$ of $\mathcal{H}$,
i.e., $Tf_{0}=0$ and $Tf_k=f_{k-1}$ for $k\geq1$. The the adjoint operator $T^{*}$
must be a foreward shift on the Markushevicz basis $\{g_{k}\}_{k=0}^{\infty}$, that
is, we have $T^{*}g_{k}=g_{k+1}$ for $k=0,1,\cdots.$
\end{lemma}
\begin{proof}
It is clear that $\{g_{k}\}_{k=0}^{\infty}$ is also a markushevicz basis. Then apply
lemma \ref{Lemma: Adjoint of a Shift on a spanning minimal seq must be a shift on G basis}.
\end{proof}

\textit{Proof of Theorem \ref{Theorem: A Cowen-Douglas oper never be a shift on basis for n ge 2}}.
Since $\mathrm{dim\,ker}(T-z_0)\geq 2$, we can take a non-zero
vector $x\in \mathrm{ker}(T-z_0)$ and $(x,g_0)=0$. If $(x,g_k)=0$,
then by lemma \ref{Lemma: Adjoint of a Shift on M basis also be a shift} we have
$$
(x,g_{k+1})=(x,T^*g_{k})=(Tx,g_{k})=(z_{0}x,g_k)=0
$$
Hence, it follows by induction that $(x,g_k)=0$ for $k\geq 0$. This
implies that $x=0$ because the sequence $\{g_{k}\}_{k=0}^{\infty}$
is total, a contradiction. $\hfill{}\Box$

\section{Shift on M-basis or ONB in the Cowen-Douglas class $B_{1}(\Omega)$}

\begin{theorem}\label{Theorem: when Cowen-Douglas operators are shifts}
Let $T \in \mathcal{B}_{1}(\Omega)$, $z_0\in \Omega$ and let $B$ be
the canonical right inverse of $T-z_{0}$. Then the following
statements are equivalent:

(1) $\bigvee_{k=0}^{\infty}\mathrm{ker}{B^*}^k=\mathcal{H}$;

(2) There exists a Markushevich basis $\{f_{k}\}_{k=0}^{\infty}$ of
$\mathcal{H}$ which satisfies $f_{0}\bot f_k$ for $k\geq1$ such that
$T-z_0$ is a backward shift on $\{f_{k}\}_{k=0}^{\infty}$;

(3) There is a positive real number $\varepsilon$ such that $B^* \in
\mathcal{B}_{1}(\mathbb{D}_{\varepsilon})$, where
$\mathbb{D}_{\varepsilon}=\{z\in \mathbb{C}:\ |z|< \varepsilon\}$
\end{theorem}
\begin{proof}
$(1)\Rightarrow (2)$. Take a unit vector $f_0\in {\rm ker}(T-z_0)$
and set $f_k=B^kf_0$. Then $T-z_0$ is a backward shift on
$\{f_{k}\}_{k=0}^{\infty}$. Since
$\mathrm{ran}B=(\mathrm{ker}(T-z_0))^\bot$, we know that $f_{0}\bot
f_k$ for $k\geq 1$. Note that
$$
\gamma(z)=f_0+\sum_{k=1}^{\infty} f_k(z-z_{0})^{k}
$$
is the canonical cross-section with the initial vector $f_0$, it
follows by Theorem \ref{Theorem: Holomorphic cross-section is
spanning} and Theorem \ref{Theorem: Coefficent sequence of
psedocanonical section is Minimal} that the sequence
$\{f_{k}\}_{k=0}^{\infty}$ is a complete and minimal. By
Hahn--Banach Theorem, there is unique sequence
$\{g_{k}\}_{k=0}^{\infty}$ in $\mathcal{H}$ with $g_0=f_0$ such that
the pair $(f_k,g_k)$ is a biorthogonal system. And, resembling to
the proof of Theorem \ref{Theorem: A Cowen-Douglas oper never be a
shift on basis for n ge 2}, we know also that $(T-z_{0})^*$ is the
forward shift on the vector sequence $\{g_{k}\}_{k=0}^{\infty}$.
Moreover, since $B^*(T-z_0)^*=I$ and
$$
\mathrm{ker}B^*=(\mathrm{ran}B)^\bot=\mathrm{ker}(T-z_0),
$$
we have that
$$
B^*g_{0}=B^*f_0=0 \hbox{ and } B^*g_{k+1}=B^*(T-z_0)^*g_{k}=g_k
\hbox{ for } k\geq 1
$$
i.e., $B^*$ is the backward shift on the sequence
$\{g_{k}\}_{k=0}^{\infty}$. Hence we have that $g_0, \cdots
,g_{k-1}$ are in $\mathrm{ker}{B^*}^k$. Now, making use of index
formulas, it follows that
$$
\mathrm{dim\,ker}{B^*}^k=\mathrm{ind}{B^*}^k=-\mathrm{ind}B^k
=\mathrm{ind}(T-z_0)^k=k.
$$
So, we know that the family $\{g_0, \cdots ,g_{k-1}\}$ is a basis of
$\mathrm{ker}{B^*}^k$. This implies that
$$
\bigvee_{k=0}^{\infty}\{g_k\}=\bigvee_{k=0}^{\infty}\mathrm{ker}{B^*}^k=\mathcal{H}.
$$
Thus, $\{f_{k}\}_{k=0}^{\infty}$ is a Markushevich basis.

$(2)\Rightarrow (3)$. As given, there is the  biorthogonal system
$(f_k,g_k)$ such that $f_{0}\bot f_k$ for $k\geq1$ and the sequence
$\{g_{k}\}_{k=0}^{\infty}$ is total. The same as in the preceding,
we know that $B^*$ is the backward shift on the sequence
$\{g_{k}\}_{k=0}^{\infty}$. Hence, it holds that
$$
\bigvee_{k=0}^{\infty}\mathrm{ker}{B^*}^k=\bigvee_{k=0}^{\infty}\{g_k\}=\mathcal{H}.
$$
Since $\mathrm{ran}B^*$ is a closed set, we have $0\in \rho _F(B^*)$
and $ \mathrm{ind}{B^*}=1$. Therefore,  There is a positive number
$\varepsilon$ with $\varepsilon <\frac{1}{r(T-z_0)}$ such that
$\mathbb{D}_{\varepsilon}\subset \rho _F(B^*)$ and $
\mathrm{ind}(B^*-z)=1$ for $z\in \mathbb{D}_{\varepsilon}$. Assume
that $(B-\overline{z})x=0$ for $z\in \mathbb{D}_{\varepsilon}$ and
$x \in\mathcal{H}$. Then $x=\overline{z}(T-z_0)x$ because $B$ is
right inverse of $T-z_{0}$. Hence, if $z=0$ then $x=0$; If $z\not=0$
then $(T-z_0-{\overline{z}}^{-1})x=0$. Since
${\overline{z}}^{-1}\notin \sigma(T-z_0)$, it follows that $x=0$.
Thus, we obtain that
$$
\mathrm{dim\,ker}(B^*-z)=\mathrm{ind}(B^*-z)=1 \hbox{ for } z\in
\mathbb{D}_{\varepsilon}.
$$
Also, we have that
$$
\mathrm{ran}(B^*-z)=(\mathrm{ker}(B-\overline{z}))^\bot=\mathcal{H}
\hbox{ for } z\in \mathbb{D}_{\varepsilon}.
$$
Now, we have proved that $B^* \in
\mathcal{B}_{1}(\Omega_{\varepsilon})$.

$(3)\Rightarrow (1)$. It is obviously.
\end{proof}

\begin{example}
Suppose $\{e_{n}\}_{n=1}^{\infty}$ is an ONB of the Hilbert space $\mathcal{H}$.
Define
$$
f_{n}=e_{n}-e_{n+1}, g_{n}=\sum_{k=1}^{n}e_{k},
$$
then $\{(e_{n}, f_{n})\}_{n=1}^{\infty}$ is a biorthogonal system and both $\{f_{n}\}_{n=1}^{\infty}$
and $\{g_{n}\}_{n=1}^{\infty}$ are Markushevicz basis. And it is not hard to see that they are not basis.
Moreover, the classical backward shift on the ONB $\{e_{n}\}_{n=1}^{\infty}$ is also a backward shift on the
markushevicz basis $\{g_{n}\}_{n=1}^{\infty}$.

\end{example}

For $T \in \mathcal{B}_{1}(\Omega)$, using the curvature function of
the bundle $E_T$,  M. J. Cowen and R. G. Douglas gave a
characterization of $T$ being a backward weighted shift on an
orthonormal basis \cite[Corollary 1.9]{Douglas1}. Now, in
terminology of operator theory, we give another characterization.

\begin{theorem}\label{Theorem: C-D operators are shifts (1)}
Let $T \in \mathcal{B}_{1}(\Omega)$, $z_0\in \Omega$ and let $B$ be
the canonical right inverse of $T-z_{0}$. Then $T-z_0$ is a backward
weighted shift on an orthonormal basis of $\mathcal{H}$ if and only
if, for every $k\geq 0$, the subspace ${\mathcal
M}_k:=B^k\mathrm{ker}(T-z_0)$ is invariant for $B^*B$.
\end{theorem}
\begin{proof}
Suppose that $T-z_0$ is a backward weighted shift on an orthonormal
basis $\{e_{k}\}_{k=0}^{\infty}$ with the weight sequence
$\{w_{k}\}_{k=1}^{\infty}$, i.e.,  $(T-z_0)e_0=0$ and
$(T-z_0)e_k=w_ke_{k-1}$ for $k>0$. Note that $T-z_0$ is in
$\mathcal{B}_{1}(\Omega_0)$, where $\Omega_0=\Omega - \{z_0\}$, it
follows from \cite[Theorem 3.2]{Li} that the sequence
$\{|w_{k}|^{-1}\}_{k=1}^{\infty}$ is bounded. It is easy to verify
that $B$ is a forward weighted shift on the orthonormal basis
$\{e_{k}\}_{k=0}^{\infty}$ with the weight sequence
$\{w_{k}^{-1}\}_{k=1}^{\infty}$,  i.e., $Be_k=w_k^{-1}e_{k+1}$ for
$k \geq 0$ and  $B^*$ is a backward weighted shift
$\{e_{k}\}_{k=0}^{\infty}$ with the weight sequence
$\{\overline{w}_{k}^{-1}\}_{k=1}^{\infty}$. So we have that
$B^*Be_k=|w_{k}|^{-2}e_k$ for $k \geq 0$. Since
$\mathrm{ker}(T-z_0)=\mathrm{span}\{e_0\}$ we know ${\mathcal M}_k$
is invariant for $B^*B$.

Conversely, if ${\mathcal M}_k$ is invariant for $B^*B$ then we take
a unit vector $f_0$ from $\mathrm{ker}(T-z_0)$ and set $f_k=B^kf_0$.
Similar to the proof of $(1)\Rightarrow (2)$ in Theorem
\ref{Theorem: Cowen-Douglas operators are shifts}, we know that the
sequence $\{f_{k}\}_{k=0}^{\infty}$ is a complete and minimal and
$T-z_0$ is a backward shift on $\{f_{k}\}_{k=0}^{\infty}$. And, by
Lemma \ref{Lemma: Existence of Canonical Right Inverse}, we have
that $f_0\perp f_k$ for $k>0$. Since $f_k \in {\mathcal M}_k$ and
$\mathrm{dim}{\mathcal M}_k=1$, it follows that
$B^*Bf_k=\lambda_kf_k$. If $(f_i,f_j)=0$ for $i<j$, then
$$
(f_{i+1},f_{j+1})=(Bf_i,Bf_j)=(B^*Bf_i,f_j)=(\lambda_if_i,f_j)=0.
$$
This implies that the sequence $\{f_{k}\}_{k=0}^{\infty}$ is
orthogonal. Set $e_k=\frac{f_{k}}{\|f_k\|}$. Then we have that the
sequence $\{e_{k}\}_{k=0}^{\infty}$ is an orthonormal basis of
$\mathcal{H}$ and $T-z_0$ is a backward weighted shift on
$\{e_{k}\}_{k=0}^{\infty}$.
\end{proof}

Imitating the proof of Theorem \ref{Theorem: C-D operators are
shifts (1)}, we can also come to the following.

\begin{theorem}\label{Theorem: C-D operators are shifts (2)}
Let $T \in \mathcal{B}_{1}(\Omega)$, $z_0\in \Omega$ and let $B$ be
the canonical right inverse of $T-z_{0}$. Then $T-z_0$ is a backward
shift on an orthonormal basis of $\mathcal{H}$ if and only if
$B^*B=I$, i.e., $B$ is an isometry.
\end{theorem}

Recall that a sequence $\{f_{k}\}_{k=0}^{\infty}$ is called a
\textit{Schauder basis} of $\mathcal{H}$ if for every vector $x \in
\mathcal{H}$ there exists a unique sequence
$\{\alpha_{k}\}_{k=0}^{\infty}$ of complex numbers such that the
series $\sum_{k=0}^{\infty} \alpha_{k}f_{k}$ converges to $x$ in
norm.

\begin{example}\label{Example: Shift on conditional basis}
Although theorem \ref{Theorem: A Cowen-Douglas oper never be a shift on basis for n ge 2}
tell us that for $n \ge 2$, a Cowen-Douglas operator never be a shift on basis, if we consider
the n-multiple shift case(or more general, operator-weighted shift) then we can get more interest examples.
Here we show that $S^{2}$ can be seen as a 2-multiple shift on some conditional basis.
Suppose $\{e_{n}\}_{n=1}^{\infty}$ is an ONB of the Hilbert space $\mathcal{H}$.
Let $\{\alpha_{n}\}_{n=1}^{\infty}$ be a sequence of positive numbers such that $\sum_{n=1}^{\infty}n\alpha_{n}^{2}<\infty$
and $\sum_{n=1}^{\infty} \alpha_{n}=\infty$.
Then the sequences $\{f_{n}\}_{n=1}^{\infty}, \{g_{n}\}_{n=1}^{\infty}$ defined as
$$\begin{array}{ll}
f_{2n-1}=e_{2n-1}+\sum_{i=n}^{\infty} \alpha_{i-n+1}e_{2i}, &~~f_{2n}=e_{2n}, ~~n=1, 2, \cdots \\
g_{2n-1}=e_{2n-1}, &~~g_{2n}=\sum_{i=n}^{\infty} -\alpha_{i-n+1}e_{2i-1}+e_{2n}, ~~n=1, 2, \cdots
\end{array}
$$
are both conditional basis of $\mathcal{H}$(see \cite{Singer1}, Example14.5, p429.).
Let $S$ be the classical foreward shift on $\{e_{n}\}_{n=1}^{\infty}$. Then we have
$$
S^{2}f_{n}=f_{n+2},  \hbox{ for }n\ge 1
$$
and
$$
(S^{*})^{2}g_{n}=g_{n-2}, \hbox{ for } n>2, \hbox{ and } (S^{*})^{2}g_{1}=(S^{*})^{2}g_{2}=0.
$$
Hence $S^{2}$ is a foreward 2-multiple shift on a conditional basis, and
$(S^{*})^{2}$ is a foreward 2-multiple shift on a conditional basis.
\end{example}

\begin{question}
When a Cowen-Douglas operator in $\mathcal{B}_{1}(\Omega)$ can be a
shift on some Schauder basis?
\end{question}

\end{document}